%% file: main.tex
\documentclass{amsart}
\usepackage[margin=1in]{geometry}

\usepackage{color,amssymb, comment, mathrsfs,tikz,upgreek,contour, soul, colonequals, mathdots,fancyvrb}
\usepackage{tikz-cd}
\usepackage{cleveref}
\usetikzlibrary{cd}
\usetikzlibrary{fit,shapes.geometric}
\usetikzlibrary{matrix}
\usetikzlibrary{arrows}
\usepackage{enumitem}
\usepackage{mathtools}
\usepackage[all]{xy}
\usepackage[mathlines,pagewise]{lineno}
\allowdisplaybreaks

\input{lst-Macaulay2.tex}

\newtheorem{theorem}{Theorem}[section]
\newtheorem{lemma}[theorem]{Lemma}
\newtheorem{proposition}[theorem]{Proposition}

\usepackage[only,llbracket,rrbracket,llparenthesis,rrparenthesis]{stmaryrd} 
\usepackage{accsupp}

\theoremstyle{definition}
\newtheorem{definition}[theorem]{Definition}

\newtheorem{notation}[theorem]{Notation}

\theoremstyle{remark}
\newtheorem{remark}[theorem]{Remark}

\newtheorem{chunk}[theorem]{}

\numberwithin{equation}{section}

\newcommand{\kk}{\Bbbk}

\newcommand{\del}{\partial}
\newcommand{\ds}{\displaystyle}

\newcommand{\xra}{\xrightarrow}

\newcommand{\Tor}{\operatorname{Tor}}

\newcommand{\q}{\mathbf{q}}
\newcommand{\m}{\mathfrak{m}}

\newcommand{\pf}[2]{\mathrm{pf}_{\overline{#1}}(#2)}

\newcommand{\ee}{\mathsf{e}}
\newcommand{\ww}{\mathsf{w}}
\newcommand{\ff}{\mathsf{f}}
\renewcommand{\gg}{\mathsf{g}}

\newcommand{\vv}{\mathsf{v}}
\newcommand{\ov}{\overline}
\newcommand{\rank}{\mathrm{rank}}

\makeatletter
\DeclareFontEncoding{LS2}{}{\@noaccents}
\makeatother
\DeclareFontSubstitution{LS2}{stix}{m}{n}

\DeclareSymbolFont{largesymbolsstix}{LS2}{stixex}{m}{n}

\DeclareMathDelimiter{\lbrbrak}{\mathopen}{largesymbolsstix}{"EE}{largesymbolsstix}{"14}
\DeclareMathDelimiter{\rbrbrak}{\mathclose}{largesymbolsstix}{"EF}{largesymbolsstix}{"15}

\crefname{diagram}{diagram}{diagrams}
\crefname{diagram}{Diagram}{Diagrams}
\creflabelformat{diagram}{(#1) #2 #3}

\usepackage{arydshln}

\newcommand\mc[1]{\multicolumn{1}{|l}{#1}} 

\begin{document}

\title[Trimming five generated Gorenstein ideals]{Trimming five generated Gorenstein ideals}

\author[L.~Ferraro]{Luigi Ferraro}

\author[W. F. Moore]{W.\ Frank Moore}


\begin{abstract}
Let $(R,\mathfrak{m},\Bbbk)$ be a regular local ring of dimension 3. Let $I$ be a Gorenstein ideal of $R$ of grade 3. It follows from a result of Buchsbaum and Eisenbud that there is a skew-symmetric matrix of odd size such that $I$ is generated by the sub-maximal pfaffians of this matrix. Let $J$ be the ideal obtained by multiplying some of the pfaffian generators of $I$ by $\mathfrak{m}$; we say that $J$ is a trimming of $I$. In a previous work, the first author and A. Hardesty constructed an explicit free resolution of $R/J$ and computed a DG algebra structure on this resolution.  They utilized these products to analyze the Tor algebra of such trimmed ideals. Missing from their result was the case where $I$ is five generated. In this paper we address this case.
\end{abstract}
\maketitle
\section{Introduction}

Let $(R,\m,\kk)$ denote a regular local ring, and let $I$ be a perfect ideal of grade $3$. In \cite{BuchEisen}, Buchsbaum and Eisenbud proved that a minimal free resolution $F_\bullet$ of $R/I$ possesses a differential graded (DG) algebra structure. This DG algebra induces a graded algebra structure on $\mathrm{H}(F_\bullet\otimes_R\kk)$; this graded algebra is denoted by $\Tor^R(R/I,\kk)$ and refered to as the Tor algebra of $I$. By results of Weyman \cite{Weyman} and of Avramov, Kustin, and Miller \cite{AKM}, this graded algebra structure does not depend on the DG algebra structure on $F_\bullet$ and it falls into one of five distinct classes $\mathbf{C,T,B,G,H}$. 

Of particular interest are ideals whose Tor algebra is of class $\mathbf{G}$. Gorenstein ideals of grade 3 that are not complete intersections belong to this class. 
Avramov conjectured in \cite{Avramov2012} that every ideal whose Tor algebra belongs to class $\mathbf{G}$ is Gorenstein. However, Christensen and Veliche provided a counterexample to this conjecture in \cite{CVExamples}. An infinite family of counterexamples was later found by Christensen, Veliche and Weyman in \cite{trimming}. The counterexamples in \cite{trimming} were constructed by a process referred to as ``trimming", which involves replacing a generator $g$ of the ideal with $\m g$.

In \cite{Alexis}, the first author and Hardesty studied the class of the Tor algebra of trimmings of grade 3 Gorenstein ideals in three-dimensional regular local rings, generalizing the main result of Christensen, Veliche and Weyman \cite[Theorem 2.4]{trimming}. Missing from their main result \cite[Theorem 5.8]{Alexis} is a study of trimmings of five generated grade 3 Gorenstein ideals who do not trim to class $\mathbf{G}$. In this paper we address this final missing case. More precisely, we prove that if $I$ is a five-generated grade 3 Gorenstein ideal in a three-dimensional regular local ring, and $J$ is the ideal obtained by trimming the first $t$ generators of $I$, if $J$ is not of class $\mathbf{G}$, then the following tables hold:

\begin{center}
\begin{tabular} {|c|c|} \hline
\mc{\quad\quad\quad$t=3$}&\\\hline
$\mathrm{Class}$ & $\mu(J)$ \\\hline
$\mathbf{B}$&$\mathrm{8}$\\\hline
$\mathbf{H}\mathrm{(1,1)}$&$\mathrm{7}$\\\hline
$\mathbf{H}\mathrm{(1,0)}$&$\mathrm{6}$\\\hline
\end{tabular}\quad\quad\quad
\begin{tabular} {|c|c|} \hline
\mc{\quad\quad\quad$t=2$}&\\\hline
$\mathrm{Class}$ & $\mu(J)$ \\\hline
$\mathbf{B}$&$\mathrm{6}$\\\hline
$\mathbf{H}\mathrm{(2,1)}$&$\mathrm{5}$\\\hline
$\mathbf{T}$&$\mathrm{4}$\\\hline
\end{tabular}
\end{center}
see \Cref{thm:t=3} and \Cref{thm:t=2} respectively; where $\mu(J)$ denotes the minimal number of generators of $J$. The case $t=1$ was already addressed in \cite[Theorem 2.4]{trimming} where it is proved that the following table holds:

\begin{center}
\begin{tabular} {|c|c|} \hline
\mc{\quad\quad\quad$t=1$}&\\\hline
$\mathrm{Class}$ & $\mu(J)$ \\\hline
$\mathbf{B}$&$\mathrm{5}$\\\hline
$\mathbf{H}\mathrm{(3,2)}$&$\mathrm{4}$\\\hline
\end{tabular}
\end{center}
while for $t=4,5$ the ideal $J$ is always of class $\mathbf{G}$, see Remark \ref{rmk:t=4,5}.

The paper is organized as follows. In Section 2 we recall background information on perfect ideals of grade 3 and their classification, on DG algebra resolutions of Gorenstein ideals, on pfaffians and we recall the previous results of the first author and Hardesty from \cite{Alexis}. In Section 3 we study the class of the Tor algebra of ideals obtained by trimming three generators of five generated grade 3 Gorenstein ideals. In Section 4 we study the class of the Tor algebra of ideals obtained by trimming two generators of five generated grade 3 Gorenstein ideals. The remaining cases (trimming 1,4 or 5 generators) have already been studied and are addressed in Section 2. In Section 5 we provide an example showing that the classes from the main results of Section 3 and 4 are realized. In the Appendix we provide the Macaulay2 code used for some of the computations in Section 4.


\section{Background and Notation}

In this paper, $(R,\m,\kk)$ represents a regular local ring of dimension 3. The results presented here remain applicable even when $R$ is a three-dimensional standard graded polynomial ring with coefficients in $\kk$.

\begin{chunk}
Let $I$ be a perfect ideal of $R$ of grade 3. We say that $I$ has \emph{format} $(1,m,m+n-1,n)$ if the minimal free resolution $F_\bullet$ of $R/I$ is of the form
\[
F_\bullet:0\longrightarrow R^n\longrightarrow R^{m+n-1}\longrightarrow R^m\longrightarrow R\longrightarrow0.
\]
We fix bases
\begin{equation}\label{basis}
\{e_i\}_{i=1,\ldots,m},\quad\{f_i\}_{i=1,\ldots,m+n-1},\quad\{g_i\}_{i=1,\ldots,n}
\end{equation}
of $F_1,F_2$ and $F_3$ respectively.

As demonstrated in \cite{BuchEisen}, any free resolution of length 3 possesses a DG algebra structure. The DG algebra structure on $F_\bullet$ gives rise to a graded $\kk$-algebra structure on $\mathrm{H}(F_\bullet\otimes_R\kk)$, denoted as $\Tor^R(R/I,\kk)$.

As established in \cite{AKM}, while the DG algebra structure of $F_\bullet$ may not be unique, the algebra structure on $\Tor^R(R/I,\kk)$ is. Furthermore, the bases in \eqref{basis} can be selected such that the induced $\kk$-bases in $\Tor^R(R/I,\kk)$, denoted as
\[
\{\ee_i\}_{i=1,\ldots,m},\quad\{\ff_i\}_{i=1,\ldots,m+n-1},\quad\{\gg_i\}_{i=1,\ldots,n},
\]
exhibit products aligned with one of the multiplicative structures detailed in the product table below:
\begin{center}
\begin{tabular}{rlrl}
    $\mathbf{C}(3)$ & $\ee_1 \ee_2 = \ff_3$, $\ee_2 \ee_3 = \ff_1$, $\ee_3 \ee_1 = \ff_2$ & $\ee_i \ff_i = \gg_1$ & for $1\leq i\leq 3$\\
    $\mathbf{T}$ & $\ee_1 \ee_2 = \ff_3$, $\ee_2 \ee_3 = \ff_1$, $\ee_3 \ee_1 = \ff_2$\\
    $\mathbf{B}$ & $\ee_1 \ee_2 = \ff_3$ & $\ee_i \ff_i = \gg_1$ & for $1\leq i \leq 2$\\
    $\mathbf{G}(r)$ &  & $\ee_i \ff_i = \gg_1$ & for $1\leq i\leq r$\\
    $\mathbf{H}(p,q)$ & $\ee_i \ee_{p+1} = \ff_i$ for $1\leq i\leq p$ & $\ee_{p+1}\ff_{p+j} = \gg_j$ & for $1\leq j\leq q$,
\end{tabular}
\end{center}
with $r,p,q$ nonnegative integers and $r\geq2$. The products not listed are either zero or can be deduced from the ones listed by graded-commutativity. Depending on the product structure of $\Tor^R(R/I,\kk)$, we say that the ideal $I$ belongs to one of the following classes: $\mathbf{C}(3)$, $\mathbf{T}$, $\mathbf{B}$, $\mathbf{G}(r)$, and $\mathbf{H}(p,q)$.
\end{chunk}

\begin{notation}
Let $T=(T_{i,j})$ be a $m\times m$ skew-symmetric matrix with entries in $R$ and with zeros on the diagonal. By $\pf{j_1,\ldots,j_n}{T}$ we denote the pfaffian of the submatrix of $T$ obtained by removing the columns and rows in positions $j_1,\ldots, j_n$. See \cite[2.2]{Alexis} for more details.
\end{notation}

\begin{chunk}
We recall that the \emph{unit step function}, also known as the \emph{Heaviside step function}, is the function $\theta:\mathbb{R}\backslash\{0\}\rightarrow\{0,1\}$ defined as
\[
\theta(x)=\begin{cases}
0\quad\mathrm{if}\;x<0\\
1\quad\mathrm{if}\;x>0.
\end{cases}
\]

Let $m$ be a positive integer and let $i,j,r$ be distinct elements of $\{1,\ldots,m\}$. We introduce the following notation

\begin{equation}\label{eq:Sigma3Theta}
\sigma_{i,j,r} \colonequals (-1)^{i+j+r+1+\theta(r-i)+\theta(r-j)+\theta(j-i)}.
\end{equation}
\end{chunk}

\begin{chunk}\label{ch:GorRes}
Let $I$ be a Gorenstein ideal of $R$ of grade 3 that is not a complete intersection. It was shown in \cite{BuchEisen} that there is a skew-symmetric matrix $T$ of odd size $m$, with zeros on the diagonal and entries in $\m$, such that $I=((-1)^{i+1}\pf{i}{T})_{i=1,\ldots,m}$. Moreover, a minimal free resolution of $R/I$ is given by 
\[
F_\bullet:0\rightarrow R\xra{D_3}R^m\xra{D_2}R^m\xra{D_1}R\rightarrow0,
\]
where $D_1=\begin{pmatrix}\pf{1}{T}&\cdots&(-1)^{i+1}\pf{i}{T}&\cdots&\pf{m}{T}\end{pmatrix}$, $D_2=T$ and $D_3=D_1^*$, the transpose of $D_1$.

A DG algebra structure on this resolution has been found in \cite{small}. Let $\{e_i\}_{i=1,\ldots,m},\{f_i\}_{i=1,\ldots,m}$ and $\{g\}$ be bases in degrees $1,2$ and $3$, respectively, of the resolution $F_\bullet$ with the differentials described by the matrices $D_1,D_2,D_3$ above.
A product on $F_\bullet$ is given by the following formulae
\begin{equation}\label{eq:ProdEE}
e_i\cdot_Fe_j=\sum_{r=1}^m\sigma_{i,j,r}\pf{i,j,r}{T}f_r,\quad\mathrm{for}\;i<j,
\end{equation}
\[
e_i\cdot_F f_j=\delta_{i,j}g,
\]
where $\delta_{i,j}$ is the Kronecker delta.
\end{chunk}

We fix a generating set for the maximal ideal $\m = (z_1,z_2,z_3)$. Let $I$ be a Gorenstein ideal of $R$ of grade 3 that is not a complete intersection and $F_\bullet$ the resolution of $I$ constructed in \Cref{ch:GorRes}. Let $I=(y_1,\ldots,y_m)$ where $y_i\colonequals(-1)^{i+1}\pf{i}{T}$ for $i=1,\ldots,m$ and let $t$ be an integer such that $1\leq t\leq m$, we denote by $J$ the ideal $y_1\m+\cdots y_t\m+(y_{t+1},\ldots,y_m)$; following the terminology set in \cite{trimming} we say that $J$ is obtained from $I$ by trimming the first $t$ generators of $I$. Before we recall the structure of a free resolution of $R/J$, we need to set some notations.

\begin{notation}
Let $(G_\bullet,\delta_\bullet)$ be the Koszul resolution of $k$ over $R$, with
\[
G_1=Ru_1\oplus Ru_2\oplus Ru_3,\quad G_2=Rv_{1,2}\oplus Rv_{1,3}\oplus Rv_{2,3},\quad G_3=Rw,
\]
\[
\delta_1=\begin{pmatrix}z_1&z_2&z_3\end{pmatrix},\quad\delta_2=\begin{pmatrix}-z_2&-z_3&0\\z_1&0&-z_3\\0&z_1&z_2\end{pmatrix},\quad\delta_3=\begin{pmatrix}z_3\\-z_2\\z_1\end{pmatrix}.
\]
We will need $t$ copies of $G_\bullet$, which we will denote by $(G_\bullet^k,\delta^k_\bullet)$ for $k=1,\ldots, t$. We denote the generators of $G_1^k$ by $u_1^k,u_2^k,u_3^k$, and similarly for $G_2^k,G_3^k$. We set $v^k_{\beta,\alpha}=-v^k_{\alpha,\beta}$ for $\alpha<\beta$ and $\alpha,\beta\in\{1,2,3\}$.
\end{notation}

\begin{notation}
Let $c_{i,j,l}$, for $i,j=1,\ldots,m$ and $l=1,2,3$, be elements of $R$ satisfying the following equality 
\begin{equation}\label{eq:Tji}
T_{j,i} = \sum_{l=1}^3 c_{i,j,l}z_l.
\end{equation}
\end{notation}

\begin{remark}
It follows from \cite[Theorem 3.1]{Alexis} that a free resolution of $R/J$, which we denote by $C_\bullet$, is of the form

\begin{center}

\begin{tikzpicture}[baseline=(current  bounding  box.center)]
 \matrix (m) [matrix of math nodes,row sep=3em,column sep=4em,minimum width=2em] {
0&F_3\oplus\left(\ds\bigoplus_{k=1}^tG_3^k\right)&F_2\oplus(\oplus_{k=1}^tG_2^k)&F_1^\prime\oplus\left(\ds\bigoplus_{k=1}^tG_1^k\right)&R.\\};
\path[->] (m-1-1) edge (m-1-2);
\path[->] (m-1-2) edge  node[above]{$\partial_3$} (m-1-3);
\path[->] (m-1-3) edge  node[above]{$\partial_2$} (m-1-4);
\path[->] (m-1-4) edge  node[above]{$\partial_1$} (m-1-5);
\end{tikzpicture}
\end{center}
The only possible units in the differentials of this resolution may appear in the bottom left corner of $\del_2$, which we denote by $-Q$. We point out that in \cite{Alexis} the matrix $Q$ was denoted by $Q_1$. Therefore, we only recall the construction of this map. The map $-Q: F_2\rightarrow\bigoplus_{k=1}^tG_1^k$ is defined by
\[
Q=\begin{pmatrix}q^1\\\vdots\\q^t\end{pmatrix},
\]
where the maps $q^k:F_2\rightarrow G_1^k$ are defined by
\[
q^k(f_i) = \sum_{l=1}^3 c_{i,k,l}u_l^k.
\]
We will denote $Q\otimes_R\kk$ by $\ov{Q}$ and $C_\bullet\otimes_R\kk$ by $\overline{C_\bullet}$.
\end{remark}

\begin{notation}
The basis elements $e_1,\ldots, e_m$ of $C_\bullet$ induce elements in $\Tor^R(R/J,\kk)$ that we denote by $\ee_1,\ldots, \ee_m$. A similar notation will be used for the remaining basis elements of $C_\bullet$.
\end{notation}

\begin{notation}\label{not:tripled}
Let $i,j,k$ be distinct elements of $[5]$, and let $\{r,h\}=[5]\backslash\{i,j,k\}$. Let $\alpha,\beta\in[3]$ with $\alpha<\beta$. Then,
\[
d_{\alpha,\beta}^{k,i,j}\colonequals\sigma_{i,j,r}\sigma_{i,j,r,h,k}\begin{vmatrix}c_{h,k,\alpha}& c_{r,k,\alpha}\\ c_{h,k,\beta} & c_{r,k,\beta}\end{vmatrix},
\]
where
\begin{align}\label{eq:Sigma5Theta}
    \sigma_{i,j,r,h,k} &= (-1)^{h+k+1+\theta(k-i)+\theta(k-j)+\theta(k-r)+\theta(k-h)+\theta(h-i)+\theta(h-j)+\theta(h-r)}.
\end{align}
We point out that the quantity defining $d_{\alpha,\beta}^{k,i,j}$ does not depend on the order of $r$ and $h$.
\end{notation}

\begin{remark}\label{rmk:DG}
A DG algebra structure on the resolution $C_\bullet$ was constructed in \cite[Theorem A.1]{Alexis}. We recall the products on this resolution whose coefficients may be units. We also make use of the proofs of  \cite[Lemma 5.4 and Proposition 5.7]{Alexis} to simplify the products in the 5 generated case. We keep the labels used in \cite[Theorem A.1]{Alexis}.

\noindent \underline{a. $F_1'\otimes F_1' \rightarrow F_2 \oplus (\oplus_{k=1}^t G_2^k)$}\\
\begin{equation*}
e_i\cdot e_j\colonequals e_i\cdot_Fe_j+\sum_{k=1}^t d^{k,i,j}_{1,2}v^k_{1,2}+d^{k,i,j}_{1,3}v^k_{1,3} + d^{k,i,j}_{2,3}v^k_{2,3},
\end{equation*}
for $t+1\leq i,j\leq 5$.

\vspace{4mm}
\noindent \underline{e. $F_1' \otimes F_2 \rightarrow F_3 \oplus \left( \oplus_{k=1}^t G_3^k \right)$}\\
\begin{equation*}
e_i \cdot f_j\colonequals \begin{cases}
e_i\cdot_Ff_j, & t+1\leq j\leq 5 \\
\ds\sigma_{i,r,h}\sigma_{i,r,h,s,j}\begin{vmatrix}c_{h,j,1} & c_{s,j,1} & c_{r,j,1} \\
c_{h,j,2} & c_{s,j,2} & c_{r,j,2}\\
c_{h,j,3} & c_{s,j,3} & c_{r,j,3}\end{vmatrix}w^j, & 1\leq j\leq t,\quad\mathrm{where}\;\{r,h,s\}=[5]\backslash\{i,j\}.
\end{cases}
\end{equation*}
for $t+1\leq i \leq 5$.
\end{remark}

\begin{notation}
For the sake of readibility we denote the class $\mathbf{H}(0,0)$ by $\mathbf{G}(0)$ and the class $\mathbf{H}(0,1)$ by $\mathbf{G}(1)$.
\end{notation}

\begin{definition}
Let $t\in[5]$ and let $T$ be a $5\times5$ skew-symmetric matrix with entries in $\m$.
We say that $T$ satisfies the \emph{$\mathbf{G}$-trimming condition for $t$} if for every $i,j,k$ distinct with $t+1\leq i,j\leq5$ and $1\leq k\leq t$, the $2\times 2$ minors of the matrix
\[
\begin{pmatrix}
\overline{c_{h,k,1}} &\overline{c_{r,k,1}}\\
\overline{c_{h,k,2}} &\overline{c_{r,k,2}}\\
\overline{c_{h,k,3}} &\overline{c_{r,k,3}}
\end{pmatrix}
\]
are zero, where $\{h,r\}=[5]\backslash\{k,i,j\}$ and the bar denotes the residue class modulo $\m$.

Let $I$ be a grade 3 Gorenstein ideal presented by the skew-symmetric matrix $T$, if there is no ambiguity we say that $I$ satisfies the $\mathbf{G}$-trimming condition for $t$ if $T$ does.
\end{definition}

\begin{remark}
We point out that $T$ satisfies the $\mathbf{G}$-trimming condition for $t$ if and only if the coefficients of $v_{1,2}^k,v_{1,3}^k,v_{2,3}^k$ in product a. in \Cref{rmk:DG} are zero modulo $\m$ for all $k=1,\ldots, 5$ and all $i,j=t+1,\ldots, 5$.
\end{remark}

\begin{remark}
It was proved in \cite[Theorem 5.8(1)]{Alexis} that a five generated Gorenstein ideal presented by the skew-symmetric matrix $T$ trims to an ideal of class $\mathbf{G}$ if and only if $T$ satisfies the $\mathbf{G}$-trimming condition.
\end{remark}

\begin{notation}
Consider the following projection map
\[
\pi:\bigoplus_{k=1}^tG_1^k\rightarrow\left(\bigoplus_{k=1}^tG_1^k\right)/\mathrm{Im}\;Q|_{f_1,\ldots,f_t},
\]
and denote by $\ov{\pi}$ the map $\pi\otimes_R\kk$. We denote by $p(T,t)$ the following
\[
\rank_\kk(\ov{\pi}\ov{Q}).
\]
We point out that $p(T,t)$ is the number of pivot columns, among the last $5-t$ columns, of the matrix $\ov{Q}$. 
\end{notation}

\begin{remark}\label{rmk:t=4,5}
If $t=5$, then the $\mathbf{G}$-trimming condition is vacuously satisfied. It follows from \cite[Theorem 5.8]{Alexis} that the trimmed ideal is of class $\mathbf{G}(0)$.

If $t=4$, then the $\mathbf{G}$-trimming condition is also vacuously satisfied. It follows from \cite[Theorem 5.8]{Alexis} that the trimmed ideal is of class $\mathbf{G}(1-p(T,4))$.
\end{remark}

\begin{remark}
Let $t=1$ and let $I$ be a Gorenstein ideal not satisfying the $\mathbf{G}$-trimming condition. Let $J$ be the trimmed ideal. It was proved in \cite[Theorem 2.4]{trimming} that the class of $J$ is
\[
\begin{cases}
\mathbf{H}(3,2)\quad\mathrm{if}\;\mu(J)=4,\\
\mathbf{B}\hphantom{(3,2)}\quad\mathrm{if}\;\mu(J)=5.
\end{cases}
\]
It follows from \cite[Theorem 5.8]{Alexis} and the proof of \cite[Corollary 5.11(1)]{Alexis} that if $t=1$, then $\mu(J)=7-p(T,1)$ and $\mathrm{rank}(\ov{Q})=p(T,1)$. Since our main results break down the possible classes of the trimmed ideal based upon the invariant $p(T,t)$ and $\mathrm{rank}(\ov{Q})$, we take this opportunity to restate \cite[Theorem 2.4]{trimming} in terms of these invariants. The class of $J$ is

\begin{equation}\label{eq:t=1}
\begin{cases}
\mathbf{H}(3,2)\quad\mathrm{if}\;p(T,1)=\mathrm{rank}(\ov{Q})=3,\\
\mathbf{B}\hphantom{(3,2)}\quad\mathrm{if}\;p(T,1)=\mathrm{rank}(\ov{Q})=2.
\end{cases}
\end{equation}
We point out that $p(T,1)\leq3$ since $\ov{Q}$ only has 3 rows. Moreover, if the $\mathbf{G}$-trimming condition is not satisfied, then $p(T,1)\geq2$.
\end{remark}
Therefore, the cases we will focus on in this paper are $t=2,3$.

\begin{notation}
Let $M$ be an $n\times n$ matrix with entries in $\kk$. Let $\{i_1,\ldots, i_k\},\{j_1,\ldots, j_l\}\subseteq[n]$ be two sets of indeces. We denote by $M_{\ov{i_1,\ldots, i_k};\ov{j_1,\ldots,j_l}}$, the submatrix of $M$ obtained by removing the rows $i_1,\ldots, i_k$ and the columns $j_1,\ldots, j_l$. We denote by $M_{i_1,\ldots, i_k;j_1,\ldots, j_l}$ the submatrix of $M$ consisting of rows $i_1,\ldots, i_k$ and columns $j_1,\ldots, j_l$.
\end{notation}

\begin{notation}
Let $I$ be a grade 3 perfect ideal. We will be denoting the algebra $\Tor^R(R/I,\kk)$ by $\mathcal{T}^I$. If there is no ambiguity we will simply be denoting it by $\mathcal{T}$.
\end{notation}

\section{Trimming three generators}
In this section $I$ will denote a Gorenstein ideal of grade 3 presented by a skew-symmetric matrix of odd size $T$, and $J$ the ideal obtained by trimming the first three generators of $I$. We will also be assuming that $I$ does not satisfy the $\mathbf{G}$-trimming condition.

\begin{remark}\label{rem:gen3}
Under these assumptions, the only nonzero product of type $(a)$ is the product $e_4e_5$. Therefore the Tor algebra of $J$ satisfies the following equality
\[
\mathrm{rank}_\kk(\mathcal{T}_1^{\hspace{0.02cm}2})=1.
\]
\end{remark}

\begin{lemma}\label{lem:gen3}
The first three columns of $\ov{Q}$ are pivot columns.
\end{lemma}
\begin{proof}
The matrix $\ov{Q}$ is of the following form
\[
\begin{pmatrix}
0& \ov{c_{2,1,1}} & \ov{c_{3,1,1}} & \ov{c_{4,1,1}} & \ov{c_{5,1,1}}\\
0& \ov{c_{2,1,2}} & \ov{c_{3,1,2}} & \ov{c_{4,1,2}} & \ov{c_{5,1,2}}\\
0& \ov{c_{2,1,3}} & \ov{c_{3,1,3}} & \ov{c_{4,1,3}} & \ov{c_{5,1,3}}\\\hdashline
-\ov{c_{2,1,1}} &0 & \ov{c_{3,2,1}} & \ov{c_{4,2,1}} & \ov{c_{5,2,1}}\\
-\ov{c_{2,1,2}} &0 & \ov{c_{3,2,2}} & \ov{c_{4,2,2}} & \ov{c_{5,2,2}}\\
-\ov{c_{2,1,3}} &0 & \ov{c_{3,2,3}} & \ov{c_{4,2,3}} & \ov{c_{5,2,3}}\\\hdashline
-\ov{c_{3,1,1}} & -\ov{c_{3,2,1}} & 0 & \ov{c_{4,3,1}} & \ov{c_{5,3,1}}\\
-\ov{c_{3,1,2}} & -\ov{c_{3,2,2}} & 0 & \ov{c_{4,3,2}} & \ov{c_{5,3,2}}\\
-\ov{c_{3,1,3}} & -\ov{c_{3,2,3}} & 0 & \ov{c_{4,3,3}} & \ov{c_{5,3,3}}
\end{pmatrix}
\]
Since $I$ does not satisfy the $\mathbf{G}$-trimming condition, there is a nonzero $2\times2$ minor in one of the following submatrices 
\[
\ov{Q}_{1,2,3;2,3},\quad\ov{Q}_{4,5,6;1,3},\quad\ov{Q}_{7,8,9;1,2}.
\]
Because of the symmetric features of the first three colums of $\ov{Q}$, the presence of a nonzero $2\times2$ minor in one of the three submatrices above forces the first three columns of $\ov{Q}$ to be pivot columns.
\end{proof}

\begin{samepage}
\begin{theorem}\label{thm:t=3}
Let $I$ be a 5 generated grade 3 Gorenstein ideal not satisfying the $\mathbf{G}$-trimming condition. Let $J$ be the ideal obtained by trimming the first three pfaffian generators of $I$. Then the class of $J$ and its format, as functions of $p(T,3)$ or $\mathrm{rank}(\ov{Q})$, are

\begin{center}
\begin{tabular} {|c|c|c|c|} \hline
$p(T,3)$ & \rule{0pt}{2.4ex}  $\mathrm{rank}(\ov{Q})$ & $\mathrm{Class}$ & $\mathrm{Format}$ \\\hline
0&3&$\mathbf{B}$&$\mathrm{(1,8,11,4)}$\\\hline
1&4&$\mathbf{H}\mathrm{(1,1)}$&$\mathrm{(1,7,10,4)}$\\\hline
2&5&$\mathbf{H}\mathrm{(1,0)}$&$\mathrm{(1,6,9,4)}$\\\hline
\end{tabular}
\end{center}
\end{theorem}
\end{samepage}
\begin{proof}
It follows from \cite[Theorem 5.8]{Alexis} that the format of $J$ is
\[
(1,11-\mathrm{rank}(\ov{Q}),14-\mathrm{rank}(\ov{Q}),4).
\]
By \Cref{lem:gen3} one has that 
\[
\rank(\ov{Q})=3+p(T,3).
\]
The previous two displays give the format of $J$ in the three cases of the theorem.

The determinants appearing in the products $\ee_i\ff_j$ with $4\leq i\leq5$ and $1\leq j\leq 3$ are the determinants of the following submatrices
\[
\ov{Q}_{1,2,3;2,3,5},\quad\ov{Q}_{4,5,6;1,3,5},\quad\ov{Q}_{7,8,9;1,2,5},
\]
\[
\ov{Q}_{1,2,3;2,3,4},\quad\ov{Q}_{4,5,6;1,3,4},\quad\ov{Q}_{7,8,9;1,2,4}.
\]

we will make use of this fact throughout the proof.

By \Cref{rem:gen3} $\mathrm{rank}_\kk(\mathcal{T}_1^{\hspace{0.02cm}2})=1$. The only classes satisfying this equality are $\mathbf{B}$ and $\mathbf{H}(1,q)$ for $q\geq0$.

Without loss of generality we can assume that $\ov{Q}$ is in its reduced row echelon form, since row operations affect basis elements that do not contribute to the product in $\mathcal{T}$.

\textbf{Case 1}: $p(T,3)=0$.

If one of the six minors above is nonzero, then by applying row operations to $\ov{Q}$ it would follow that one of the last two columns is a pivot column, contradicting the assumption $p(T,3)=0$. Therefore all products of the form $\ee_i\ff_j$ with $4\leq i\leq5$ and $1\leq j\leq3$ are zero in $\ov{C_\bullet}$.

Since only the first three columns of $\ov{Q}$ are pivot columns, one can make a change of basis of the form
\[
\ff_j'=\ff_j+\alpha_j\ff_1+\beta_j\ff_2+\gamma_j\ff_3,\quad j=1,\ldots,5;\alpha_j,\beta_j,\gamma_j\in\kk,
\]
to split off the nonminimal part of $\ov{C_\bullet}$. The only basis elements of this form in $\mathcal{T}_2$ are $\ff_4'$ and $\ff_5'$. To show that the class of $J$ in this case is $\mathbf{B}$ it only remains to notice that
\[
\ee_j\ff_j'=\ee_j\ff_j=\gg\quad\mathrm{for}\;j=4,5,
\]
while the other products of basis elements of degree 1 by basis elements of degree 2 are zero.

\textbf{Case 2}: $p(T,3)=1$.

We assume that the fourth column is a pivot column, the case where the fifth column is a pivot column is similar. As in the previous case, one can make a change of basis of the form

\[
\ff_j'=\ff_j+\alpha_j\ff_1+\beta_j\ff_2+\gamma_j\ff_3+\delta_j\ff_4,\quad j=1,\ldots,5;\alpha_j,\beta_j,\gamma_j,\delta_j\in\kk,
\]
to split off the nonminimal part of $\ov{C_\bullet}$. The only basis element of this form in $\mathcal{T}_2$ is $\ff_5'$. We notice that $\ee_5\ff_5'\neq0$. We set $\ee_4'=\ee_4-\delta_5\ee_5$, and notice that $\ee_4',\ee_5$ and $\ee_4,\ee_5$ are bases of the same vector space. Moreover, $\ee_4'\ff_5'=0$. This shows that the class of $J$ is $\mathbf{H}(1,1)$.

\textbf{Case 3}: $p(T,3)=2$.

In this case every column of $\ov{Q}$ is a pivot column. Therefore, after an appropriate change of basis, the elements $\ff_1',\ldots,\ff_5'$ are split off from the minimal resolution of $R/J$. Since there is no element of $\mathcal{T}_2$ with a nonzero product with an element of $\mathcal{T}_1$, it follows that the class of $J$ must be $\mathbf{H}(1,0)$.
\end{proof}
\section{Trimming two generators}
In this section $I$ will denote a Gorenstein ideal of grade 3, and $J$ the ideal obtained by trimming the first two generators of $I$. We will also be assuming that $I$ does not satisfy the $\mathbf{G}$-trimming condition.

\begin{lemma}\label{lem:gen2}
The first two columns of $\ov{Q}$ are pivot columns.
\end{lemma}
\begin{proof}
The matrix $\ov{Q}$ is of the following form
\[
\begin{pmatrix}
0& \ov{c_{2,1,1}} & \ov{c_{3,1,1}} & \ov{c_{4,1,1}} & \ov{c_{5,1,1}}\\
0& \ov{c_{2,1,2}} & \ov{c_{3,1,2}} & \ov{c_{4,1,2}} & \ov{c_{5,1,2}}\\
0& \ov{c_{2,1,3}} & \ov{c_{3,1,3}} & \ov{c_{4,1,3}} & \ov{c_{5,1,3}}\\\hdashline
-\ov{c_{2,1,1}} &0 & \ov{c_{3,2,1}} & \ov{c_{4,2,1}} & \ov{c_{5,2,1}}\\
-\ov{c_{2,1,2}} &0 & \ov{c_{3,2,2}} & \ov{c_{4,2,2}} & \ov{c_{5,2,2}}\\
-\ov{c_{2,1,3}} &0 & \ov{c_{3,2,3}} & \ov{c_{4,2,3}} & \ov{c_{5,2,3}}
\end{pmatrix}
\]
Since $I$ does not satisfy the $\mathbf{G}$-trimming condition, one of the following six submatrices of $\ov{Q}$ must have a nonzero $2\times2$ minor
\[
\ov{Q}_{1,2,3;2,3},\quad\ov{Q}_{1,2,3;2,4},\quad\ov{Q}_{1,2,3;2,5},
\]
\[
\ov{Q}_{4,5,6;1,3},\quad\ov{Q}_{4,5,6;1,4},\quad\ov{Q}_{4,5,6;1,5}.
\]

Because of the symmetric features of the first two colums of $\ov{Q}$, the presence of a nonzero $2\times2$ minor in one of the six submatrices above forces the first two columns of $\ov{Q}$ to be pivot columns.
\end{proof}

\begin{remark}
It follows from the previous proof that $p(T,2)\geq1$.
\end{remark}

\begin{notation}
We will be denoting by $E$ the following $6\times3$ matrix
\[E\colonequals
\begin{pmatrix}
\ov{d_{1,2}^{1,3,4}} & \ov{d_{1,2}^{1,3,5}} & \ov{d_{1,2}^{1,4,5}}\\[6pt]
\ov{d_{1,3}^{1,3,4}} & \ov{d_{1,3}^{1,3,5}} & \ov{d_{1,3}^{1,4,5}}\\[6pt]
\ov{d_{2,3}^{1,3,4}} & \ov{d_{2,3}^{1,3,5}} & \ov{d_{2,3}^{1,4,5}}\\[6pt]
\ov{d_{1,2}^{2,3,4}} & \ov{d_{1,2}^{2,3,5}} & \ov{d_{1,2}^{2,4,5}}\\[6pt]
\ov{d_{1,3}^{2,3,4}} & \ov{d_{1,3}^{2,3,5}} & \ov{d_{1,3}^{2,4,5}}\\[6pt]
\ov{d_{2,3}^{2,3,4}} & \ov{d_{2,3}^{2,3,5}} & \ov{d_{2,3}^{2,4,5}}
\end{pmatrix}
\]
where the entries have been defined in \Cref{not:tripled}. We point out that the columns correspond to the coordinates of $\ee_3\ee_4,\ee_3\ee_5,\ee_4\ee_5$ with respect to the basis $\vv_{1,2}^1,\vv_{1,3}^1,\vv_{2,3}^1,\vv_{1,2}^2,\vv_{1,3}^2,\vv_{2,3}^2$.
\end{notation}

The next three lemmas can be proved straightforwardly by computing minors of matrices, however these computations are very long and tedious. For the sake of the reader's sanity (and the writers'), we provide, in the appendix, Macaulay2 code to check the identities displayed in the lemmas.
\begin{lemma}\label{lem:3x3E}
The following equalities hold

\begin{align*}
|E_{\ov{1,2,5};}|=\ov{c_{2,1,2}}|\ov{Q}_{\ov{1};}|,\quad\quad\quad & |E_{\ov{1,2,4};}|=\ov{c_{2,1,3}}|\ov{Q}_{\ov{1};}|,\\
|E_{\ov{1,3,6};}|=\ov{c_{2,1,1}}|\ov{Q}_{\ov{2};}|, \quad\quad\quad& |E_{\ov{1,3,4};}|=\ov{c_{2,1,3}}|\ov{Q}_{\ov{2};}|, \\
|E_{\ov{2,3,6};}|=\ov{c_{2,1,1}}|\ov{Q}_{\ov{3};}|, \quad\quad\quad& |E_{\ov{2,3,5};}|=\ov{c_{2,1,2}}|\ov{Q}_{\ov{3};}|, \\
|E_{\ov{2,4,5};}|=-\ov{c_{2,1,2}}|\ov{Q}_{\ov{4};}|, \quad\quad\quad& |E_{\ov{1,4,5};}|=-\ov{c_{2,1,3}}|\ov{Q}_{\ov{4};}|, \\
|E_{\ov{3,4,6};}|=-\ov{c_{2,1,1}}|\ov{Q}_{\ov{5};}|, \quad\quad\quad& |E_{\ov{1,4,6};}|=-\ov{c_{2,1,3}}|\ov{Q}_{\ov{5};}|, \\
|E_{\ov{3,5,6};}|=-\ov{c_{2,1,1}}|\ov{Q}_{\ov{6};}|, \quad\quad\quad& |E_{\ov{2,5,6};}|=-\ov{c_{2,1,2}}|\ov{Q}_{\ov{6};}|.
\end{align*}

\end{lemma}


The remaining $3\times3$ minors of $E$ are given by the following:

\begin{lemma}
\label{lem:3x3EPart2}
The following equalities hold

\begin{minipage}{.45\textwidth}
\begin{eqnarray*}
|E_{\ov{1,2,3};}|& = & 0, \\
|E_{\ov{1,2,6};}|& = & \ov{c_{2,1,1}}|\ov{Q}_{\ov{1};}|, \\
|E_{\ov{3,4,5};}|& = & -\ov{c_{2,1,1}}|\ov{Q}_{\ov{4};}|, \\
|E_{\ov{2,3,4};}|& = & -\ov{c_{2,1,1}}|\ov{Q}_{\ov{1};}|+\ov{c_{2,1,2}}|\ov{Q}_{\ov{2};}|,\\
\end{eqnarray*}
\end{minipage}
\begin{minipage}{.45\textwidth}
\begin{eqnarray*}
|E_{\ov{4,5,6};}| & = & 0,\\
|E_{\ov{1,3,5};}| & = & \ov{c_{2,1,2}}|\ov{Q}_{\ov{2};}|,\\
|E_{\ov{2,4,6};}| & = & -\ov{c_{2,1,2}}|\ov{Q}_{\ov{5};}|,\\
|E_{\ov{1,5,6};}| & = & \ov{c_{2,1,1}}|\ov{Q}_{\ov{4};}|-\ov{c_{2,1,2}}|\ov{Q}_{\ov{5};}|.\\
\end{eqnarray*}
\end{minipage}
\end{lemma}


\begin{lemma}\label{lem:2x2E}
Up to signs, the set of $2\times2$ minors of $E$ and the set of $4\times4$ minors of $\ov{Q}$ involving the first two columns coincide.
\end{lemma}


\begin{theorem}\label{thm:t=2}
Let $I$ be a 5 generated grade 3 Gorenstein ideal not satisfying the $\mathbf{G}$-trimming condition. Let $J$ be the ideal obtained by trimming the first two pfaffian generators of $I$. Then the class of $J$ and its format, as functions of $p(T,2)$ or $\mathrm{rank}(\ov{Q})$, are

\begin{center}
\begin{tabular} {|c|c|c|c|} \hline
$p(T,2)$ & \rule{0pt}{2.4ex}$\mathrm{rank}(\ov{Q})$ &$\mathrm{Class}$ & $\mathrm{Format}$ \\\hline
1&3&$\mathbf{B}$&$\mathrm{(1,6,8,3)}$\\\hline
2&4&$\mathbf{H}\mathrm{(2,1)}$&$\mathrm{(1,5,7,3)}$\\\hline
3&5&$\mathbf{T}$&$\mathrm{(1,4,6,3)}$\\\hline
\end{tabular}
\end{center}
\end{theorem}
\begin{proof}
It follows from \cite[Theorem 5.8]{Alexis} that the format of $J$ is
\[
(1,9-\mathrm{rank}(\ov{Q}),11-\mathrm{rank}(\ov{Q}),3).
\]
By \Cref{lem:gen2} one has that 
\[
\rank(\ov{Q})=2+p(T,2).
\]
The previous two displays give the format of $J$ in the three cases of the theorem.

Without loss of generality we can assume that $\ov{Q}$ is in its reduced row echelon form, since row operations affect basis elements that do not contribute to the product in $\mathcal{T}$.

\textbf{Case 1}: $p(T,2)=1$.

In this case $\mathrm{rank}(\ov{Q})=3$ and therefore all the $4\times4$ minors of $\ov{Q}$ are zero. By \Cref{lem:2x2E}, it follows that the $2\times2$ minors of $E$ are zero, and therefore
\[
\mathrm{rank}_\kk(\mathcal{T}_1^{\hspace{0.02cm}2})=1.
\]
We assume that the third column of $\ov{Q}$ is a pivot column, the other two cases are similar. 
After an appropriate change of basis
\[
\ff_i'=\ff_i+\alpha_{1,i}\ff_1+\alpha_{2,i}\ff_2+\alpha_{3,i}\ff_3,\quad\alpha_{1,i},\alpha_{2,i},\alpha_{3,i}\in\kk\;\mathrm{for}\;i=1,2,3,4,5,
\]
one can split off the nonminimal part of $\overline{C_\bullet}$.

Therefore the only two elements of $\mathcal{T}_2$ that may have nonzero products with elements of $\mathcal{T}_1$ are $\ff_4'$ and $\ff_5'$. 

We set $\ee_3'=\ee_3-\alpha_{3,4}\ee_4-\alpha_{3,5}\ee_5$. Since $\ee_1$ and $\ee_2$ have a product of zero when multiplied with $\ff_4'$ or $\ff_5'$, we study the products between $\ee_3',\ee_4,\ee_5$ and $\ff_4',\ff_5'$.

We claim that $\ee_4\ff_4'=\ee_5\ff_5'=\gg$. To show this, it suffices to show that 
\[
\ee_4\ff_1=\ee_4\ff_2=\ee_4\ff_3=\ee_5\ff_1=\ee_5\ff_2=\ee_5\ff_3=0.
\]
The products $\ee_4\ff_3$ and $\ee_5\ff_3$ are zero by definition of the products e. in \Cref{rmk:DG}. The coefficient appearing in the product $\ee_4\ff_1$ is, up to a sign, the determinant of the matrix $\ov{Q}_{\ov{4,5,6};\ov{1,4}}$. If this determinant is nonzero, then either the fourth or fifth column of $\ov{Q}$ would be a pivot column, a contradiction. Therefore $\ee_4\ff_1=0$. A similar argument works for the remaining equalities.
This also shows that $\ee_4\ff_5'=\ee_5\ff_4'=0$. It remains to notice that
\[
\ee_3'\ff_4'=\alpha_{3,4}\gg-\alpha_{3,4}\gg=0,
\]
since one can argue that $\ee_3\ff_1=\ee_3\ff_2=0$ as above. A similar argument shows that $\ee_3'\ff_5'=0$. This proves that $J$ is of class $\mathbf{B}$.

\textbf{Case 2}: $p(T,2)=2$.

To prove that $J$ is of class $\mathbf{H}(2,1)$ it suffices to show that
\[
\mathrm{\rank}_\kk(\mathcal{T}_1^{\hspace{0.02cm}2})=2,\quad\mathrm{and}\quad\mathrm{rank}_\kk(\mathcal{T}_1\mathcal{T}_2)=1.
\]

Since in this case $\rank(\ov{Q})=4$, by \Cref{lem:3x3E} and \Cref{lem:3x3EPart2} the $3\times3$ minors of $E$ are zero. Moreover, by \Cref{lem:2x2E} there is a nonzero $2\times2$ minor of $E$ because if $\rank(\ov{Q})=4$, then there has to be a $4\times4$ minor, involving the pivot columns, that is nonzero, and by \Cref{lem:gen2} we know that the first two columns of $\ov{Q}$ are pivot columns. This implies that
\[
\mathrm{rank}_\kk(\mathcal{T}_1^{\hspace{0.02cm}2})=2.
\]

To prove $\mathrm{rank}_\kk(\mathcal{T}_1\mathcal{T}_2)=1$ we assume that the fifth column is not a pivot column, the remaining two cases are similar. Since the first four columns of $\ov{Q}$ are pivot columns, and the fifth one is not, the following system of equations has a unique solution
\[
\begin{pmatrix}
0& \ov{c_{2,1,1}} & \ov{c_{3,1,1}} & \ov{c_{4,1,1}} \\
0& \ov{c_{2,1,2}} & \ov{c_{3,1,2}} & \ov{c_{4,1,2}} \\
0& \ov{c_{2,1,3}} & \ov{c_{3,1,3}} & \ov{c_{4,1,3}} \\\hdashline
-\ov{c_{2,1,1}} &0 & \ov{c_{3,2,1}} & \ov{c_{4,2,1}} \\
-\ov{c_{2,1,2}} &0 & \ov{c_{3,2,2}} & \ov{c_{4,2,2}} \\
-\ov{c_{2,1,3}} &0 & \ov{c_{3,2,3}} & \ov{c_{4,2,3}} 
\end{pmatrix}\begin{pmatrix}\alpha_1\\\alpha_2\\\alpha_3\\\alpha_4\end{pmatrix}=\begin{pmatrix}
 \ov{c_{5,1,1}}\\
 \ov{c_{5,1,2}}\\
\ov{c_{5,1,3}}\\\hdashline
 \ov{c_{5,2,1}}\\
 \ov{c_{5,2,2}}\\
 \ov{c_{5,2,3}}
\end{pmatrix}.
\]
Set
\[
\ff_5'=\ff_5-\alpha_1\ff_1-\alpha_2\ff_2-\alpha_3\ff_3-\alpha_4\ff_4.
\]
In the basis $\ff_1,\ldots,\ff_4,\ff_5'$, the linear transformation defined by the matrix $\ov{Q}$ has the zero column as the last column. A further change of the basis elements $\ff_1,\ldots,\ff_4$ allows us to split off the nonminimal part of $\ov{C_\bullet}$. Set
\[
\ee_4'=\ee_4+\alpha_4\ee_5,\quad \ee_3'=\ee_3+\alpha_3\ee_5.
\]
Clearly $\ee_5\ff_5'\neq0$. To finish the proof we show that $\ee_3'\ff_5'=\ee_4'\ff_5'=0$. We check the last equality, the remaining one is similar.

Taking the product $\ee_4'\ff_5'$ yields
\begin{align*}
\alpha_1\begin{vmatrix}\ov{c_{2,1,1}}&\ov{c_{3,1,1}}&\ov{c_{5,1,1}}\\
\ov{c_{2,1,2}}&\ov{c_{3,1,2}}&\ov{c_{5,1,2}}\\
\ov{c_{2,1,3}}&\ov{c_{3,1,3}}&\ov{c_{5,1,3}}\end{vmatrix}\ww^1-\alpha_2\begin{vmatrix}-\ov{c_{2,1,1}}&\ov{c_{3,2,1}}&\ov{c_{5,2,1}}\\
-\ov{c_{2,1,2}}&\ov{c_{3,2,2}}&\ov{c_{5,2,2}}\\
-\ov{c_{2,1,3}}&\ov{c_{3,2,3}}&\ov{c_{5,2,3}}\end{vmatrix}\ww^2-\alpha_4\gg+\\
-\alpha_4\alpha_1\begin{vmatrix}\ov{c_{2,1,1}}&\ov{c_{3,1,1}}&\ov{c_{4,1,1}}\\
\ov{c_{2,1,2}}&\ov{c_{3,1,2}}&\ov{c_{4,1,2}}\\
\ov{c_{2,1,3}}&\ov{c_{3,1,3}}&\ov{c_{4,1,3}}\end{vmatrix}\ww^1+\alpha_4\alpha_2\begin{vmatrix}-\ov{c_{2,1,1}}&\ov{c_{3,2,1}}&\ov{c_{4,2,1}}\\
-\ov{c_{2,1,2}}&\ov{c_{3,2,2}}&\ov{c_{4,2,2}}\\
-\ov{c_{2,1,3}}&\ov{c_{3,2,3}}&\ov{c_{4,2,3}}\end{vmatrix}\ww^2+\alpha_4\gg.
\end{align*}
The $\gg$ component of this product obviously vanishes. We prove that the $\ww^1$ component vanishes as well, the proof of the vanishing of the $\ww^2$ component is similar. We can assume that $\alpha_1\neq0$. We first show that the two determinants are zero simultaneously.

Before proceeding, we point out that the following equality holds
\begin{equation}\label{eq:AugMat}
\begin{pmatrix}
\ov{c_{2,1,1}}&\ov{c_{3,1,1}}&\ov{c_{4,1,1}}\\
\ov{c_{2,1,2}}&\ov{c_{3,1,2}}&\ov{c_{4,1,2}}\\
\ov{c_{2,1,3}}&\ov{c_{3,1,3}}&\ov{c_{4,1,3}}
\end{pmatrix}\begin{pmatrix}\alpha_2\\\alpha_3\\\alpha_4\end{pmatrix}=\begin{pmatrix}\ov{c_{5,1,1}}\\\ov{c_{5,1,2}}\\\ov{c_{5,1,3}}\end{pmatrix}.
\end{equation}
If the determinant of the matrix of coefficients of \eqref{eq:AugMat} is zero then, the last row of the row reduced echelon form of the augmented matrix of this system is zero. This shows that
\[
\begin{vmatrix}\ov{c_{2,1,1}}&\ov{c_{3,1,1}}&\ov{c_{5,1,1}}\\
\ov{c_{2,1,2}}&\ov{c_{3,1,2}}&\ov{c_{5,1,2}}\\
\ov{c_{2,1,3}}&\ov{c_{3,1,3}}&\ov{c_{5,1,3}}\end{vmatrix}=0
\]
as well. The reverse implication is similar. If these two determinants are nonzero, then we need to show that
\[
\alpha_4=\frac{\quad\begin{vmatrix}\ov{c_{2,1,1}}&\ov{c_{3,1,1}}&\ov{c_{5,1,1}}\\
\ov{c_{2,1,2}}&\ov{c_{3,1,2}}&\ov{c_{5,1,2}}\\
\ov{c_{2,1,3}}&\ov{c_{3,1,3}}&\ov{c_{5,1,3}}\end{vmatrix}\quad}{\begin{vmatrix}\ov{c_{2,1,1}}&\ov{c_{3,1,1}}&\ov{c_{4,1,1}}\\
\ov{c_{2,1,2}}&\ov{c_{3,1,2}}&\ov{c_{4,1,2}}\\
\ov{c_{2,1,3}}&\ov{c_{3,1,3}}&\ov{c_{4,1,3}}\end{vmatrix}}.
\]
This follows by applying Cramer's rule to the system \eqref{eq:AugMat}.

\textbf{Case 3}: $p(T,2)=3$.

Since in this case $\mathrm{rank}(\ov{Q})=5$, by \Cref{lem:3x3E} it follows that $E$ has a nonzero $3\times3$ minor. Indeed, if for example $|\ov{Q}_{\ov{1};}|\neq0$, then either $\ov{c_{2,1,2}}\neq0$ or $\ov{c_{2,1,3}}\neq0$, otherwise if they were both zero, then $\ov{Q}_{\ov{1};}$ would have a zero column. So either $|E_{\ov{1,2,5};}|\neq0$ or $|E_{\ov{1,2,4};}|\neq0$.
This shows that the products $\ee_1\ee_2,\ee_1\ee_3,\ee_2\ee_3$ are linearly independent, which implies that
\[
\mathrm{rank}_\kk(\mathcal{T}_1^{\hspace{0.02cm}2})=3.
\]
Since all the columns of $\ov{Q}$ are pivot columns it follows that after an appropriate change of basis 
\[
\{\ff_1,\ldots,\ff_5\}\rightarrow\{\ff_1',\ldots,\ff_5'\},
\]
the basis elements $\ff_1',\ldots,\ff_5'$ are split off from the minimal resolution of $R/J$, therefore
\[
\mathrm{rank}_\kk(\mathcal{T}_1\mathcal{T}_2)=0.
\]
This is the product structure of a Tor algebra of class $\mathbf{T}$.
\end{proof}

\section{Realizability}
In this section we show that the eight cases from \Cref{thm:t=3}, \Cref{thm:t=2} and \eqref{eq:t=1} are realized. We provide a single ideal which realizes all cases. Let $R$ be the ring $\mathbb{F}_2[[x,y,z]]$. We identify $z_1,z_2,z_3$ from the previous section with $x,y,z$ respectively. Consider the matrix
\[
T=\begin{pmatrix}
0&y+z&z&z+y^2&z\\
y+z&0&x&x+y+z&x+y+z+z^2\\
z&x&0&z&x+z\\
z+y^2&x+y+z&z&0&x^2\\
z&x+y+z+z^2&x+z&x^2&0
\end{pmatrix},
\]
then one has
\[
\pf{1}{T}=x^3+z^3+x^2+xy+xz,\quad\pf{2}{T}=xy^2+x^2z+y^2z+xz,\quad\pf{3}{T}=y^2z^2+x^2y+xy^2+y^3+x^2z+y^2z+z^3
\]
\[
\pf{4}{T}=z^3+xy+xz,\quad\pf{5}{T}=xy^2.
\]
Let $I$ be the ideal generated by these elements.
\begin{proposition}
The ideal $I$ is a Gorenstein ideal of grade 3 minimally generated by five elements.
\end{proposition}
\begin{proof}
By \cite[Theorem 2.1]{BuchEisen} it suffices to show that the ideal $I$ has grade 3. One can check that the following equalities hold
\begin{align*}
x^2&=\frac{1}{1+x}(\pf{1}{T}+\pf{4}{T}),\\
y^5&=(y^2z+y^2)\pf{2}{T}+y^2\pf{3}{T}+y^2\pf{4}{T}+(y^2z+xz^2+xy+z^2+y)\pf{5}{T},\\
z^6&=\frac{1}{1+z^4}((xy^3+x^2z+xz)\pf{2}{T}+(z^7+xyz^4+xz^5+x^2y^2z+x^2z^3+x^3y+x^3z+z^3+xy+xz)\pf{4}{T}\\
&+(xy^3+y^3z+x^2z^2+x^3+x^2z+xyz+xz^2+xz+z^2+x)\pf{5}{T}).
\end{align*}
Therefore the elements $x^2,y^5,z^6$ form a regular sequence in $I$ of length 3.
\end{proof}

\begin{theorem}
The classes of the following trimmings of $I$ are
\begin{center}
\begin{tabular}{|c|c|}
\hline
\quad\quad$\mathrm{Generators\;trimmed}$\quad\quad  & \quad$\mathrm{Class}$\quad\quad  \\
\hline \hline
$\pf{1}{T},\pf{2}{T},\pf{4}{T}$ &$\mathbf{B}$\\\hline
 $\pf{1}{T},\pf{2}{T},\pf{3}{T}$&$\mathbf{H}(1,1)$\\\hline
 $\pf{3}{T},\pf{4}{T},\pf{5}{T}$&$\mathbf{H}(1,0)$\\ \hline
 $\pf{1}{T},\pf{2}{T}$&$\mathbf{B}$\\\hline
 $\pf{3}{T},\pf{4}{T}$&$\mathbf{H}(2,1)$\\\hline
 $\pf{3}{T},\pf{5}{T}$&$\mathbf{T}$\\\hline
 $\pf{1}{T}$&$\mathbf{B}$\\\hline
 $\pf{5}{T}$&$\mathbf{H}(3,2)$\\\hline
\end{tabular}
\end{center}
\end{theorem}

\begin{proof}
We first prove that if one trims the generators $\pf{1}{T},\pf{2}{T}$ and $\pf{3}{T}$, then one obtains an ideal of class $\mathbf{H}(1,1)$. The matrix $\ov{Q}$ associated to this trimming is
\[
\ov{Q}=\begin{pmatrix}
0&0&0&0&0\\
0&1&0&0&0\\
0&1&1&1&1\\\hdashline
0&0&1&1&1\\
1&0&0&1&1\\
1&0&0&1&1\\\hdashline
0&1&0&0&1\\
0&0&0&0&0\\
1&0&0&1&1
\end{pmatrix}.
\]
The $\mathbf{G}$-trimming condition is not satisfied since the determinant of the submatrix $\ov{Q}_{2,3;2,3}$ is nonzero. It is elementary to see that the rank of this matrix is 4, and therefore by \Cref{thm:t=3} the trimmed ideal is of class $\mathbf{H}(1,1)$.

We next show that if one trims the generators $\pf{3}{T},\pf{4}{T}$ and $\pf{5}{T}$, then one obtains an ideal of class $\mathbf{H}(1,0)$. We need to present the ideal $I$ with a matrix $T'$ such that
\[
\pf{1}{T'}=\pf{3}{T},\quad \pf{2}{T'}=\pf{4}{T},\quad\pf{3}{T'}=\pf{5}{T}.
\]
By \cite[Remark 5.12]{Alexis}, to find $T'$ one can conjugate $T$ by the permutation matrix associated with the permutation $(1, 4, 2, 5,3)$. Let
\[
P\colonequals\begin{pmatrix}
0&0&0&1&0\\
0&0&0&0&1\\
1&0&0&0&0\\
0&1&0&0&0\\
0&0&1&0&0
\end{pmatrix},
\]
then the matrix $T'\colonequals PTP^{-1}$ has the desired property. A computation shows that
\[
T'=\begin{pmatrix}
0&z&x+z&z&x\\
z&0&x^2&z+y^2&x+y+z\\
x+z&x^2&0&z&x+y+z+z^2\\
z&z+y^2&z&0&y+z\\
x&x+y+z&x+y+z+z^2&y+z&0
\end{pmatrix}.
\]
The matrix $\ov{Q'}$ associated to $T'$ is
\[
\ov{Q'}=\begin{pmatrix}
0&0&1&0&1\\
0&0&0&0&0\\
0&1&1&1&0\\\hdashline
0&0&0&0&1\\
0&0&0&0&1\\
1&0&0&1&1\\\hdashline
1&0&0&0&1\\
0&0&0&0&1\\
1&0&0&1&1
\end{pmatrix}.
\]
The $\mathbf{G}$-trimming condition is not satisfied since the determinant of $\ov{Q'}_{1,3;2,3}$ is nonzero. An elementary computation shows that the rank of this matrix is 5, therefore by \Cref{thm:t=3} the trimmed ideal is of class $\mathbf{H}(1,0)$. The remaining cases are proved similarly.
\end{proof}


\appendix

\section{Macaulay2 Code}

In this appendix, we include the code used to check the claims made
in Lemmas \ref{lem:3x3E}, \ref{lem:3x3EPart2} and \ref{lem:2x2E}.  We begin 
by defining the unit step function as well as the functions \texttt{sigma3}
and \texttt{sigma5} which implement the $\sigma$ coefficients with three
and five subscripts, respectively:
\begin{lstlisting}[language=Macaulay2]
unitStep = n -> if (n > 0) then 1 else 0
sigma3 = (i,j,r) -> (-1)^(i + j + r + 1 + unitStep(r-i) + unitStep(r-j) + unitStep(j-i))
sigma5 = (i,j,r,h,k) -> (-1)^(h + k + 1 + unitStep(k-i) + unitStep(k-j) + 
                              unitStep(k-r) + unitStep(k-h) + unitStep(h-i) +
                              unitStep(h-j) + unitStep(h-r))
\end{lstlisting}
Using these functions, we can define the entries of the matrix $E$
as in Notation \ref{not:tripled}:
\begin{lstlisting}[language=Macaulay2]
entryInE = (Q,k,i,j,alpha,beta) -> (
   Q' := if (k == 1) then Q^{0,1,2} else Q^{3,4,5};
   rAndh := toList (set toList (1..5) - set {k,i,j});
   r := first rAndh;
   h := last rAndh;
   Q'' := Q'_{r-1,h-1};
   Q''' = Q''^{alpha-1,beta-1};
   sigma3(i,j,r) * sigma5(i,j,r,h,k) * det(Q''')
)
\end{lstlisting}
Finally, we use the following function to define $E$ in terms of $\ov{Q}$:
\begin{lstlisting}[language=Macaulay2]
getEMatrix = Q -> (
   tuples := { {({1,3,4},{1,2}), ({1,3,5},{1,2}), ({1,4,5},{1,2})},
               {({1,3,4},{1,3}), ({1,3,5},{1,3}), ({1,4,5},{1,3})},
	       {({1,3,4},{2,3}), ({1,3,5},{2,3}), ({1,4,5},{2,3})},
	       {({2,3,4},{1,2}), ({2,3,5},{1,2}), ({2,4,5},{1,2})},
               {({2,3,4},{1,3}), ({2,3,5},{1,3}), ({2,4,5},{1,3})},
	       {({2,3,4},{2,3}), ({2,3,5},{2,3}), ({2,4,5},{2,3})} };
   matrix applyTable(tuples, (p,q) -> entryInE(Q,p#0,p#1,p#2,q#0,q#1))
)
\end{lstlisting}
To verify the claims made in Lemmas \ref{lem:3x3E} and \ref{lem:3x3EPart2},
one defines a ring with variables corresponding to the entries of $\ov{Q}$,
and creates $E$ using \texttt{getEMatrix} as above, and then one checks
the desired equalities as in the last line given below.
\begin{lstlisting}[language=Macaulay2]
R=ZZ[x_211,x_212,x_213,x_311,x_312,x_313,
     x_411,x_412,x_413,x_511,x_512,x_513,
     x_321,x_322,x_323,x_421,x_422,x_423,
     x_521,x_522,x_523]
Q=matrix{{0,x_211,x_311,x_411,x_511},
         {0,x_212,x_312,x_412,x_512},
	 {0,x_213,x_313,x_413,x_513},
	 {-x_211,0,x_321,x_421,x_521},
	 {-x_212,0,x_322,x_422,x_522},
	 {-x_213,0,x_323,x_423,x_523}}
E = getEMatrix Q
det E^{0,4,5} == -x_211*det(Q^{1,2,3,4,5}) + x_212*det(Q^{0,2,3,4,5})
\end{lstlisting}
To verify Lemma \ref{lem:2x2E}, we compute the list of all specified
$4\times 4$ minors of $\ov{Q}$, as well as the $2 \times 2$ minors of $E$.
We check that all leading coefficients are $\pm 1$, make them monic,
remove duplicates, sort the lists, and then compare them for equality:
\begin{lstlisting}[language=Macaulay2]
cols = select(subsets({0,1,2,3,4},4), p -> member(0,p) and member(1,p))
rows = subsets({0,1,2,3,4,5},4)
minorIndices = rows ** cols
minors4Q = apply(minorIndices, p -> det Q^(p_0)_(p_1));
assert all(minors4Q, f -> leadCoefficient f == 1 or leadCoefficient f == -1)
minors4Q = sort unique apply(minors4Q, f -> f * (leadCoefficient(f))^(-1));
minors2E = (minors(2,E))_*;
assert all(minors2E, f -> leadCoefficient f == 1 or leadCoefficient f == -1)
minors2E = sort unique apply(minors2E, f -> f * (leadCoefficient(f))^(-1));
minors4Q == minors2E
\end{lstlisting}


\bibliographystyle{amsplain}
\bibliography{biblio}
\end{document}

%% file: lst-Macaulay2.tex
\usepackage{listings}
\usepackage{xcolor}
\definecolor{maccolor}{rgb}{0.3,0.3,0.8}
\lstdefinelanguage{Macaulay2}
{
basicstyle={\ttfamily},
keywordstyle={\color{maccolor!80!black}},
commentstyle={\color{gray}},
stringstyle={\color{red!40!black}},
rulecolor=\color{maccolor},
basewidth={1.2ex}, 
sensitive=false,
morecomment=[l]{--},
morecomment=[s]{-*}{*-},
morestring=[b]",
escapechar={`},
escapebegin={\rmfamily},
morekeywords={about,abs,AbstractToricVarieties,accumulate,Acknowledgement,acos,acosh,acot,addCancelTask,addDependencyTask,addEndFunction,addHook,AdditionalPaths,addStartFunction,addStartTask,Adjacent,adjoint,AdjointIdeal,AffineVariety,AfterEval,AfterNoPrint,AfterPrint,agm,AInfinity,alarm,AlgebraicSplines,Algorithm,Alignment,all,AllCodimensions,allowableThreads,ambient,analyticSpread,Analyzer,AnalyzeSheafOnP1,ancestor,ancestors,ANCHOR,and,andP,AngleBarList,ann,annihilator,antipode,any,append,applicationDirectory,applicationDirectorySuffix,apply,applyKeys,applyPairs,applyTable,applyValues,apropos,argument,Array,arXiv,Ascending,ascii,asin,asinh,ass,assert,associatedGradedRing,associatedPrimes,AssociativeAlgebras,AssociativeExpression,atan,atan2,atEndOfFile,Authors,autoload,AuxiliaryFiles,backtrace,Bag,Bareiss,baseFilename,BaseFunction,baseName,baseRing,baseRings,BaseRow,BasicList,basis,BasisElementLimit,Bayer,BeforePrint,beginDocumentation,BeginningMacaulay2,Benchmark,benchmark,Bertini,BesselJ,BesselY,betti,BettiCharacters,BettiTally,between,BGG,BIBasis,Binary,BinaryOperation,Binomial,binomial,BinomialEdgeIdeals,Binomials,BKZ,BlockMatrix,BLOCKQUOTE,BODY,Body,BoijSoederberg,BOLD,Book3264Examples,Boolean,BooleanGB,borel,Boxes,BR,break,Browse,Bruns,cache,CacheExampleOutput,CacheFunction,CacheTable,cacheValue,CallLimit,cancelTask,capture,catch,Caveat,CC,CDATA,ceiling,Center,centerString,Certification,ChainComplex,chainComplex,ChainComplexExtras,ChainComplexMap,ChainComplexOperations,ChangeMatrix,char,CharacteristicClasses,characters,charAnalyzer,check,CheckDocumentation,chi,Chordal,class,Classic,clean,clearAll,clearEcho,clearOutput,close,closeIn,closeOut,ClosestFit,CODE,code,codim,CodimensionLimit,coefficient,CoefficientRing,coefficientRing,coefficients,Cofactor,CohenEngine,CohenTopLevel,CoherentSheaf,CohomCalg,cohomology,coimage,CoincidentRootLoci,coker,cokernel,collectGarbage,columnAdd,columnate,columnMult,columnPermute,columnRankProfile,columnSwap,combine,Command,commandInterpreter,commandLine,COMMENT,commonest,commonRing,comodule,CompactMatrix,compactMatrixForm,CompiledFunction,CompiledFunctionBody,CompiledFunctionClosure,Complement,complement,complete,CompleteIntersection,CompleteIntersectionResolutions,Complexes,ComplexField,components,compose,compositions,compress,concatenate,conductor,ConductorElement,cone,Configuration,ConformalBlocks,conjugate,connectionCount,Consequences,Constant,Constants,constParser,content,continue,contract,Contributors,ConvexInterface,conwayPolynomial,ConwayPolynomials,copy,copyDirectory,copyFile,copyright,Core,CorrespondenceScrolls,cos,cosh,cot,CotangentSchubert,cotangentSheaf,coth,cover,coverMap,cpuTime,createTask,Cremona,csc,csch,current,currentColumnNumber,currentDirectory,currentFileDirectory,currentFileName,currentLayout,currentLineNumber,currentPackage,currentString,currentTime,Cyclotomic,Database,Date,DD,dd,deadParser,debug,debugError,DebuggingMode,debuggingMode,debugLevel,DecomposableSparseSystems,Decompose,decompose,deepSplice,Default,default,defaultPrecision,Degree,degree,degreeLength,DegreeLift,DegreeLimit,DegreeMap,DegreeOrder,DegreeRank,Degrees,degrees,degreesMonoid,degreesRing,delete,demark,denominator,Dense,Density,Depth,depth,Descending,Descent,Describe,describe,Description,det,determinant,DeterminantalRepresentations,DGAlgebras,diagonalMatrix,diameter,Dictionary,dictionary,dictionaryPath,diff,DiffAlg,difference,dim,directSum,disassemble,discriminant,dismiss,Dispatch,distinguished,DIV,Divide,divideByVariable,DivideConquer,DividedPowers,Divisor,DL,Dmodules,do,doc,docExample,docTemplate,document,DocumentTag,Down,drop,DT,dual,eagonNorthcott,EagonResolution,echoOff,echoOn,EdgeIdeals,edit,EigenSolver,eigenvalues,eigenvectors,eint,EisenbudHunekeVasconcelos,elapsedTime,elapsedTiming,elements,Eliminate,eliminate,Elimination,EliminationMatrices,EllipticCurves,EllipticIntegrals,else,EM,Email,End,end,endl,endPackage,Engine,engineDebugLevel,EngineRing,EngineTests,entries,EnumerationCurves,environment,Equation,EquivariantGB,erase,erf,erfc,error,errorDepth,euler,EulerConstant,eulers,even,EXAMPLE,ExampleFiles,ExampleItem,examples,ExampleSystems,Exclude,exec,exit,exp,expectedReesIdeal,expm1,exponents,export,exportFrom,exportMutable,Expression,expression,Ext,extend,ExteriorIdeals,ExteriorModules,exteriorPower,Factor,factor,false,Fano,FastMinors,FastNonminimal,FGLM,File,fileDictionaries,fileExecutable,fileExists,fileExitHooks,fileLength,fileMode,FileName,FilePosition,fileReadable,fileTime,fileWritable,fillMatrix,findFiles,findHeft,FindOne,findProgram,findSynonyms,FiniteFittingIdeals,First,first,firstkey,FirstPackage,fittingIdeal,flagLookup,FlatMonoid,flatten,flattenRing,Flexible,flip,floor,flush,fold,FollowLinks,for,forceGB,fork,FormalGroupLaws,Format,format,formation,FourierMotzkin,FourTiTwo,fpLLL,frac,fraction,FractionField,frames,FrobeniusThresholds,from,fromDividedPowers,fromDual,Function,FunctionApplication,FunctionBody,functionBody,FunctionClosure,FunctionFieldDesingularization,fusePairs,futureParser,GaloisField,Gamma,gb,GBDegrees,gbRemove,gbSnapshot,gbTrace,gcd,gcdCoefficients,gcdLLL,GCstats,genera,GeneralOrderedMonoid,GenerateAssertions,generateAssertions,generator,generators,Generic,GenericInitialIdeal,genericMatrix,genericSkewMatrix,genericSymmetricMatrix,gens,genus,get,getc,getChangeMatrix,getenv,getGlobalSymbol,getNetFile,getNonUnit,getPrimeWithRootOfUnity,getSymbol,getWWW,GF,gfanInterface,Givens,GKMVarieties,GLex,Global,global,globalAssign,globalAssignFunction,GlobalAssignHook,globalAssignment,globalAssignmentHooks,GlobalDictionary,GlobalHookStore,globalReleaseFunction,GlobalReleaseHook,Gorenstein,GradedLieAlgebras,GradedModule,gradedModule,GradedModuleMap,gradedModuleMap,gramm,GraphicalModels,GraphicalModelsMLE,Graphics,graphIdeal,graphRing,Graphs,Grassmannian,GRevLex,GroebnerBasis,groebnerBasis,GroebnerBasisOptions,GroebnerStrata,GroebnerWalk,groupID,GroupLex,GroupRevLex,GTZ,Hadamard,handleInterrupts,HardDegreeLimit,hash,HashTable,hashTable,HEAD,HEADER1,HEADER2,HEADER3,HEADER4,HEADER5,HEADER6,HeaderType,Heading,Headline,Heft,heft,Height,height,help,Hermite,hermite,Hermitian,HH,hh,HigherCIOperators,HighestWeights,Hilbert,hilbertFunction,hilbertPolynomial,hilbertSeries,HodgeIntegrals,hold,Holder,Hom,homeDirectory,HomePage,Homogeneous,Homogeneous2,homogenize,homology,homomorphism,HomotopyLieAlgebra,hooks,horizontalJoin,HorizontalSpace,HR,HREF,HTML,html,httpHeaders,Hybrid,HyperplaneArrangements,Hypertext,hypertext,HypertextContainer,HypertextParagraph,icFracP,icFractions,icMap,icPIdeal,id,Ideal,ideal,idealizer,identity,if,IgnoreExampleErrors,ii,image,imaginaryPart,IMG,ImmutableType,importFrom,in,incomparable,Increment,independentSets,indeterminate,IndeterminateNumber,Index,index,indexComponents,IndexedVariable,IndexedVariableTable,indices,inducedMap,inducesWellDefinedMap,InexactField,InexactFieldFamily,InexactNumber,InfiniteNumber,infinity,info,InfoDirSection,infoHelp,Inhomogeneous,input,Inputs,insert,installAssignmentMethod,installedPackages,installHilbertFunction,installMethod,installMinprimes,installPackage,InstallPrefix,instance,instances,IntegralClosure,integralClosure,integrate,IntermediateMarkUpType,interpreterDepth,intersect,intersectInP,Intersection,intersection,interval,InvariantRing,inverse,InverseMethod,inversePermutation,Inverses,inverseSystem,InverseSystems,Invertible,InvolutiveBases,irreducibleCharacteristicSeries,irreducibleDecomposition,isAffineRing,isANumber,isBorel,isCanceled,isCommutative,isConstant,isDirectory,isDirectSum,isEmpty,isField,isFinite,isFinitePrimeField,isFreeModule,isGlobalSymbol,isHomogeneous,isIdeal,isInfinite,isInjective,isInputFile,isIsomorphism,isLinearType,isListener,isLLL,isMember,isModule,isMonomialIdeal,isNormal,isOpen,isOutputFile,isPolynomialRing,isPrimary,isPrime,isPrimitive,isPseudoprime,isQuotientModule,isQuotientOf,isQuotientRing,isReady,isReal,isReduction,isRegularFile,isRing,isSkewCommutative,isSorted,isSquareFree,isStandardGradedPolynomialRing,isSubmodule,isSubquotient,isSubset,isSupportedInZeroLocus,isSurjective,isTable,isUnit,isWellDefined,isWeylAlgebra,ITALIC,Iterate,Jacobian,jacobian,jacobianDual,Jets,Join,join,Jupyter,K3Carpets,K3Surfaces,Keep,KeepFiles,KeepZeroes,ker,kernel,kernelLLL,kernelOfLocalization,Key,keys,Keyword,Keywords,kill,koszul,Kronecker,KustinMiller,LABEL,last,lastMatch,LATER,LatticePolytopes,Layout,lcm,leadCoefficient,leadComponent,leadMonomial,leadTerm,Left,left,length,LengthLimit,letterParser,Lex,LexIdeals,LI,Licenses,LieTypes,lift,liftable,Limit,limitFiles,limitProcesses,Linear,LinearAlgebra,LinearTruncations,lineNumber,lines,LINK,linkFile,List,list,listForm,listLocalSymbols,listSymbols,listUserSymbols,LITERAL,LLL,LLLBases,lngamma,load,loadDepth,LoadDocumentation,loadedFiles,loadedPackages,loadPackage,Local,local,localDictionaries,LocalDictionary,localize,LocalRings,locate,log,log1p,LongPolynomial,lookup,lookupCount,LowerBound,LUdecomposition,M0nbar,M2CODE,Macaulay2Doc,makeDirectory,MakeDocumentation,makeDocumentTag,MakeHTML,MakeInfo,MakeLinks,makePackageIndex,MakePDF,makeS2,Manipulator,map,MapExpression,MapleInterface,markedGB,Markov,MarkUpType,match,mathML,Matrix,matrix,MatrixExpression,Matroids,max,maxAllowableThreads,maxExponent,MaximalRank,maxPosition,MaxReductionCount,MCMApproximations,member,memoize,memoizeClear,memoizeValues,MENU,merge,mergePairs,META,method,MethodFunction,MethodFunctionBinary,MethodFunctionSingle,MethodFunctionWithOptions,methodOptions,methods,midpoint,min,minExponent,mingens,mingle,minimalBetti,MinimalGenerators,MinimalMatrix,minimalPresentation,minimalPresentationMap,minimalPresentationMapInv,MinimalPrimes,minimalPrimes,minimalReduction,Minimize,minimizeFilename,MinimumVersion,minors,minPosition,minPres,minprimes,Minus,minus,Miura,MixedMultiplicity,mkdir,mod,Module,module,ModuleDeformations,modulo,MonodromySolver,Monoid,monoid,MonoidElement,Monomial,MonomialAlgebras,monomialCurveIdeal,MonomialIdeal,monomialIdeal,MonomialIntegerPrograms,MonomialOrbits,MonomialOrder,Monomials,monomials,MonomialSize,monomialSubideal,moveFile,multidegree,multidoc,multigraded,MultigradedBettiTally,MultiGradedRationalMap,multiplicity,MultiplicitySequence,MultiplierIdeals,MultiplierIdealsDim2,MultiprojectiveVarieties,mutable,MutableHashTable,mutableIdentity,MutableList,MutableMatrix,mutableMatrix,NAGtypes,Name,nanosleep,Nauty,NautyGraphs,NCAlgebra,NCLex,needs,needsPackage,Net,net,NetFile,netList,new,newClass,newCoordinateSystem,NewFromMethod,newline,NewMethod,newNetFile,NewOfFromMethod,NewOfMethod,newPackage,newRing,nextkey,nextPrime,nil,NNParser,NoetherianOperators,NoetherNormalization,NonAssociativeProduct,NonminimalComplexes,nonspaceAnalyzer,NoPrint,norm,normalCone,Normaliz,NormalToricVarieties,not,Nothing,notify,notImplemented,NTL,null,nullaryMethods,nullhomotopy,nullParser,nullSpace,Number,number,NumberedVerticalList,numcols,numColumns,numerator,numeric,NumericalAlgebraicGeometry,NumericalCertification,NumericalImplicitization,NumericalLinearAlgebra,NumericalSchubertCalculus,numericInterval,NumericSolutions,numgens,numRows,numrows,odd,oeis,of,ofClass,OL,OldPolyhedra,OldToricVectorBundles,on,OneExpression,OnlineLookup,OO,oo,ooo,oooo,openDatabase,openDatabaseOut,openFiles,openIn,openInOut,openListener,OpenMath,openOut,openOutAppend,operatorAttributes,Option,OptionalComponentsPresent,optionalSignParser,Options,options,OptionTable,optP,or,Order,order,OrderedMonoid,orP,OutputDictionary,Outputs,override,pack,Package,package,PackageCitations,PackageDictionary,PackageExports,PackageImports,PackageTemplate,packageTemplate,pad,pager,PairLimit,pairs,PairsRemaining,PARA,Parametrization,parent,Parenthesize,Parser,Parsing,part,Partition,partition,partitions,parts,path,pdim,peek,PencilsOfQuadrics,Permanents,permanents,permutations,pfaffians,PHCpack,PhylogeneticTrees,pi,PieriMaps,pivots,PlaneCurveSingularities,plus,poincare,poincareN,Points,polarize,poly,Polyhedra,Polymake,PolynomialRing,Posets,Position,position,positions,PositivityToricBundles,POSIX,Postfix,Power,power,powermod,PRE,Precision,precision,Prefix,prefixDirectory,prefixPath,preimage,prepend,presentation,pretty,primaryComponent,PrimaryDecomposition,primaryDecomposition,PrimaryTag,PrimitiveElement,Print,print,printerr,printingAccuracy,printingLeadLimit,printingPrecision,printingSeparator,printingTimeLimit,printingTrailLimit,printString,printWidth,processID,Product,product,ProductOrder,profile,profileSummary,Program,programPaths,ProgramRun,Proj,Projective,ProjectiveHilbertPolynomial,projectiveHilbertPolynomial,ProjectiveVariety,promote,protect,Prune,prune,PruneComplex,pruningMap,Pseudocode,pseudocode,pseudoRemainder,Pullback,PushForward,pushForward,Python,QQ,QQParser,QRDecomposition,QthPower,Quasidegrees,QuaternaryQuartics,QuillenSuslin,quit,Quotient,quotient,quotientRemainder,QuotientRing,Radical,radical,RadicalCodim1,radicalContainment,RaiseError,random,RandomCanonicalCurves,RandomComplexes,RandomCurves,RandomCurvesOverVerySmallFiniteFields,RandomGenus14Curves,RandomIdeals,randomKRationalPoint,RandomMonomialIdeals,randomMutableMatrix,RandomObjects,RandomPlaneCurves,RandomPoints,RandomSpaceCurves,Range,rank,RationalMaps,RationalPoints,RationalPoints2,ReactionNetworks,read,readDirectory,readlink,readPackage,RealField,RealFP,realPart,realpath,RealQP,RealQP1,RealRoots,RealRR,RealXD,recursionDepth,recursionLimit,Reduce,reducedRowEchelonForm,reduceHilbert,reductionNumber,ReesAlgebra,reesAlgebra,reesAlgebraIdeal,reesIdeal,References,ReflexivePolytopesDB,regex,regexQuote,registerFinalizer,regSeqInIdeal,Regularity,regularity,relations,RelativeCanonicalResolution,relativizeFilename,Reload,remainder,RemakeAllDocumentation,remove,removeDirectory,removeFile,removeLowestDimension,reorganize,replace,RerunExamples,res,reshape,ResidualIntersections,ResLengthThree,Resolution,resolution,ResolutionsOfStanleyReisnerRings,restart,Result,resultant,Resultants,return,returnCode,Reverse,reverse,RevLex,Right,right,Ring,ring,RingElement,RingFamily,ringFromFractions,RingMap,rootPath,roots,rootURI,rotate,round,rowAdd,RowExpression,rowMult,rowPermute,rowRankProfile,rowSwap,RR,RRi,rsort,run,RunDirectory,RunExamples,RunExternalM2,runHooks,runLengthEncode,runProgram,same,saturate,Saturation,scan,scanKeys,scanLines,scanPairs,scanValues,schedule,schreyerOrder,Schubert,Schubert2,SchurComplexes,SchurFunctors,SchurRings,SCRIPT,scriptCommandLine,ScriptedFunctor,SCSCP,searchPath,sec,sech,SectionRing,SeeAlso,seeParsing,SegreClasses,select,selectInSubring,selectVariables,SelfInitializingType,SemidefiniteProgramming,Seminormalization,separate,SeparateExec,separateRegexp,Sequence,sequence,Serialization,serialNumber,Set,set,setEcho,setGroupID,setIOExclusive,setIOSynchronized,setIOUnSynchronized,setRandomSeed,setup,setupEmacs,sheaf,SheafExpression,sheafExt,sheafHom,SheafOfRings,shield,ShimoyamaYokoyama,short,show,showClassStructure,showHtml,showStructure,showTex,showUserStructure,SimpleDoc,simpleDocFrob,SimplicialComplexes,SimplicialDecomposability,SimplicialPosets,SimplifyFractions,sin,singularLocus,sinh,size,size2,SizeLimit,SkewCommutative,SlackIdeals,sleep,SLnEquivariantMatrices,SLPexpressions,SMALL,smithNormalForm,solve,someTerms,Sort,sort,sortColumns,SortStrategy,source,SourceCode,SourceRing,SPACE,SpaceCurves,SPAN,span,SparseMonomialVectorExpression,SparseResultants,SparseVectorExpression,Spec,SpechtModule,SpecialFanoFourfolds,specialFiber,specialFiberIdeal,SpectralSequences,splice,splitWWW,sqrt,SRdeformations,stack,stacksProject,Standard,standardForm,standardPairs,StartWithOneMinor,stashValue,StatePolytope,StatGraphs,status,stderr,stdio,step,StopBeforeComputation,stopIfError,StopWithMinimalGenerators,Strategy,String,STRONG,StronglyStableIdeals,STYLE,Style,style,SUB,sub,SubalgebraBases,sublists,submatrix,submatrixByDegrees,Subnodes,subquotient,SubringLimit,Subscript,subscript,SUBSECTION,subsets,substitute,substring,subtable,Sugarless,Sum,sum,SumOfTwists,SumsOfSquares,SUP,super,SuperLinearAlgebra,Superscript,superscript,support,SVD,SVDComplexes,switch,SwitchingFields,sylvesterMatrix,Symbol,symbol,SymbolBody,symbolBody,SymbolicPowers,symlinkDirectory,symlinkFile,symmetricAlgebra,symmetricAlgebraIdeal,symmetricKernel,SymmetricPolynomials,symmetricPower,synonym,SYNOPSIS,syz,Syzygies,SyzygyLimit,SyzygyMatrix,SyzygyRows,syzygyScheme,TABLE,Table,table,take,Tally,tally,tan,TangentCone,tangentCone,tangentSheaf,tanh,target,Task,taskResult,TateOnProducts,TD,temporaryFileName,tensor,tensorAssociativity,TensorComplexes,terminalParser,terms,TEST,Test,testExample,testHunekeQuestion,TestIdeals,TestInput,tests,TEX,tex,TeXmacs,texMath,Text,TH,then,Thing,ThinSincereQuivers,ThreadedGB,threadVariable,Threshold,throw,Time,time,times,timing,TITLE,TO,to,TO2,toAbsolutePath,toCC,toDividedPowers,toDual,toExternalString,toField,TOH,toList,toLower,top,top,topCoefficients,Topcom,topComponents,topLevelMode,Tor,TorAlgebra,Toric,ToricInvariants,ToricTopology,ToricVectorBundles,toRR,toRRi,toSequence,toString,TotalPairs,toUpper,TR,trace,transpose,TriangularSets,Tries,Trim,trim,Triplets,Tropical,true,Truncate,truncate,truncateOutput,Truncations,try,TSpreadIdeals,TT,tutorial,Type,TypicalValue,typicalValues,UL,ultimate,unbag,uncurry,Undo,undocumented,uniform,uninstallAllPackages,uninstallPackage,Unique,unique,Units,Unmixed,unsequence,unstack,Up,UpdateOnly,UpperTriangular,URL,urlEncode,Usage,use,UseCachedExampleOutput,UseHilbertFunction,UserMode,userSymbols,UseSyzygies,utf8,utf8check,validate,value,values,Variable,VariableBaseName,Variables,Variety,variety,vars,Vasconcelos,Vector,vector,VectorExpression,VectorFields,VectorGraphics,Verbose,Verbosity,Verify,VersalDeformations,versalEmbedding,Version,version,VerticalList,VerticalSpace,viewHelp,VirtualResolutions,VirtualTally,VisibleList,Visualize,wait,WebApp,wedgeProduct,weightRange,Weights,WeylAlgebra,WeylGroups,when,whichGm,while,width,wikipedia,Wrap,wrap,WrapperType,XML,xor,youngest,zero,ZeroExpression,zeta,ZZ,ZZParser}
}
\lstalias{Macaulay2output}{Macaulay2}